\documentclass[12pt,letterpaper]{amsart}

\usepackage{accents}
\usepackage{enumitem}
\usepackage{graphicx}
\usepackage{caption}
\usepackage{amsmath}
\usepackage{float}
\usepackage{amsaddr}
\newtheorem{theorem}{Theorem}[section]
\newtheorem{lemma}[theorem]{Lemma}
\theoremstyle{definition}
\newtheorem{definition}[theorem]{Definition}

\newtheorem{corollary}[theorem]{Corollary}
\newtheorem{proposition}[theorem]{Proposition}
\theoremstyle{remark}

\newtheorem{claim}[theorem]{{\bf Claim}}

\numberwithin{equation}{section}

\usepackage{color}

\begin{document}

\def\C{\mathbb C}
\def\R{\mathbb R}
\def\X{\mathbb X}
\def\Z{\mathbb Z}
\def\Y{\mathbb Y}
\def\Z{\mathbb Z}
\def\N{\mathbb N}
\def\cal{\mathcal}
\def\cD{\cal D}
\def\tD{\tilde{{\cal D}}}
\def\F{\cal F}
\def\tf{\tilde{f}}
\def\tg{\tilde{g}}
\def\tu{\tilde{u}}

\def\cal{\mathcal}
\def\b{\mathcal B}
\def\c{\mathcal C}
\def\cc{\mathbb C}
\def\x{\mathbb X}
\def\r{\mathbb R} 
\def\T{\mathbb T}
\def\uu{(U(t,s))_{t\ge s}}
\def\vv{(V(t,s))_{t\ge s}}
\def\xx{(X(t,s))_{t\ge s}}
\def\yy{(Y(t,s))_{t\ge s}}
\def\zz{(Z(t,s))_{t\ge s}}
\def\ss{(S(t))_{t\ge 0}}
\def\tt{(T(t,s))_{t\ge s}}
\def\rr{(R(t))_{t\ge 0}}
\title[Diffusive Delay in SIR Models]{\textbf{Traveling Waves for a Diffusive SIR Epidemic model with Delay in the Diffusion Term  }}

\author{William Kyle Barker}
\address{Department of Mathematics, University of Arkansas at Little Rock, 
Little Rock, AR 72204}
\email{wkbarker@ualr.edu}

\thanks{}

\date{\today}
\begin{abstract}
This paper is concerned with traveling waves to an diffusive SIR model with delay placed in the diffusion terms as well as nonlinear incidence rate with delay. Using a cross iteration scheme and partial monotone conditions it will be shown that the existence of quasi upper and lower solutions is a sufficient condition for the existence of a traveling wave front. This will be shown via Schauder's fixed point theorem. Given an appropriate basic reproduction number the traveling wave front will flow from a disease-free steady state to endemic steady state. The construction of quasi upper and lower solutions will be carried out for a specific model.   
\end{abstract}

\maketitle
\section{Introduction}
\noindent Kermack and Mckendrick studied an ordinary differential system 
\begin{align}
&S'(t)=F(S,I)(t) \\
&I'(t)=G(S,I)(t)\nonumber\\
&R'(t)=H(I,R)(t)\nonumber.
\end{align}
Here $S(t)$ is the number of individuals who
are susceptible to the disease, $I(t)$  is the number of individuals either are infected or can spread the disease, and $R(t)$ represent the number of individuals who have
been infected and then removed from the possibility of being
infected again, \cite{Kermack}. 

\medskip 
\noindent In terms of disease and the nature of it spreading, it is required to study spatial diffusion. Most models incorporate a random walk process known as Fickian Diffusion, see \cite{fif, Fisher, Kawasaki, Kolmogorov, Martcheva, Murray1, Murray2, Okubo}. The spread of disease can be modeled by traveling waves. In fact, realistic SIR models often incorporate diffusion terms and time delay in order to be more accurate. Due to the fact that there is interaction between all three populations, so it makes sense to study the dynamic structure based upon immediate past migration factors. This means addition of small delay in each diffusion term is an interesting topic to study, see \cite{Cosner, Levin}.    

\medskip 
\noindent
Cooke, \cite{Cooke} studied the relationship of the force of infection was given by the incident rate $\beta S(t)I(t-\tau),$ where $\tau$ is some positive delay and $\beta$ is a positive infection constant. 

\medskip 
\noindent 
Solutions of functional parabolic equations were first established by Schaaf, \cite{sch} using phase space
analysis and the maximum principle. Traveling wave solutions of reaction diffusion equations with delay in the reaction term, were first studied using the monotone iteration method via upper and lower solutions by Wu and Zou, \cite{wuzou}. The results were extended by Ma, \cite{ma} by using the idea of super and sub solutions and Schauder's fixed point theorem. The construction of upper and lower solutions was relaxed by Boumenir, and Nguyen, \cite{boumin}  by introducing the concept of quasi upper and lower solutions. It was also shown that the monotone iteration method adds $C^2$ smoothness the quasi upper and lower solutions, which become upper and lower solutions, respectively.

\medskip 
\noindent 
When interactive terms are introduced into models such as Lotka-Volterra or SIR models, the monotone conditions used in in Wu and Zou, \cite{wuzou} fail. Thus, the monotone iteration method, often referred as a cross or mixed
iteration has to be introduced. The definition of such an iteration is dependent upon the iterations of the groups being studied, see \cite{GanXuYang, Huazou2, lilin,lvwang,pan,Feng,Ruan,Wangzhou,WangWu,YangLiang} and their references therein. 

\medskip 
\noindent 
Traveling waves of Diffusive SIR models with certain  incident rates have been studied in \cite{Ai, Bai, Fu, GanXuYang,LiLi, Shu, WengZhao,YangLiang}. Yang, Liang, Zhang \cite{YangLiang} studied the existence of traveling waves for nonlinear terms that satisfied a certain partial quasimonotone condition using a cross iteration method for the following system  
\begin{align}
&\frac{\partial S(x,t)}{\partial t}=D_S\frac{\partial^2 S(x,t)}{\partial x^2}+B-\mu_1S(x,t)-\frac{\beta S(x,t)I(x, t-\tau)}{1+\alpha I(x, t-\tau) }, \\
&\frac{\partial I(x,t)}{\partial t}=D_I\frac{\partial^2 I(x,t)}{\partial x^2}+\frac{\beta S(x,t)I(x, t-\tau)}{1+\alpha I(x, t-\tau) }-(\mu_2+\gamma)I(x,t) \nonumber\\
&\frac{\partial R(x,t)}{\partial t}=D_R\frac{\partial^2 R(x,t)}{\partial x^2}+\gamma I(x,t) -\mu_3 R(x,t).\nonumber
\end{align}
with initial conditions, bounded domain and smooth boundary. 

\medskip
\noindent This paper is concerned with a similar model as found in Yang, Liang, Zhang \cite{YangLiang}. This project will study the dynamics when small delays are placed in each of the diffusion terms as well as a small delay in the saturation term. 

\medskip
\noindent
 The constants $D_S, D_I, D_R,  B, \mu_1, \mu_2, \mu_3,  \gamma, \alpha, \beta$ are positive. The values $D_S, D_I, D_R$ are the diffusion constants, $ \mu_1, \mu_2, \mu_3$ are the respective death rates for the populations, and $\gamma$ is the recovery rate of the infected population. The value $B$ is the natural birth rate, $\beta$ is the average number of contacts for infected individuals per day and $\tau>0$ is a latency value where infected individuals can spread the disease after that period. 

\medskip
\noindent
The exact question that we aim to answer is on the existence of traveling waves of the following system:
\begin{align}\label{SIR1}
&\frac{\partial S(x,t)}{\partial t}=D_S\frac{\partial^2 S(x,t-\tau_1)}{\partial x^2}+B-\mu_1S(x,t)-\frac{\beta S(x,t)I(x, t-\tau_4)}{1+\alpha I(x, t-\tau_4) }, \\
&\frac{\partial I(x,t)}{\partial t}=D_I\frac{\partial^2 I(x,t-\tau_2)}{\partial x^2}+\frac{\beta S(x,t)I(x, t-\tau_4)}{1+\alpha I(x, t-\tau_4) }-(\mu_2+\gamma)I(x,t) \nonumber\\
&\frac{\partial R(x,t)}{\partial t}=D_R\frac{\partial^2 R(x,t-\tau_3)}{\partial x^2}+\gamma I(x,t) -\mu_3 R(x,t).\nonumber
\end{align}

The values $\tau_1, \tau_2, \tau_4$ are migration delays and $\tau_3$ is the latency period of the disease. We assume that $\tau_1, \tau_2, \tau_3>\tau_3.$
\medskip
\noindent
The basic reproduction number $\mathcal{R}_0={B\beta}/{\mu_1(\mu_2+\gamma)}$ determines if the system has a positive steady endemic state. This happens when $\mathcal{R}_0>1.$ 

\medskip
\noindent
The existence of traveling waves for reaction diffusion equations with delay in the diffusion term using quasi upper and lower solutions via the monotone iteration method in Barker and Nguyen, \cite{BarkNguy}. Furthermore, Barker, \cite{Barker} was able to show the existence of traveling wave  for reaction diffusion systems with delay in the diffusion term using quasi upper and lower solutions via the cross iteration method.

\medskip
\noindent
The paper will be organized as follows: in Section 2 we will put forth preliminary information that will be utilized throughout the paper. In Section 3, we will show the existence of traveling waves for the system (\ref{SIR1}) via Schauder's fixed point theorem in a Banach space equipped with exponential decay norm. This will be aided by partial monotone conditions that will be placed on appropriate upper and lower solutions. 

\medskip
\noindent
It will also be shown that the  existence of super and sub solutions imply  existence of quasi upper and lower solutions, which in turn imply the existence of smooth upper and lower solutions. This will also imply the existence of quasi upper and lower solutions is a sufficient condition for the existence of traveling waves.  In Section 4, we will construct explicit super and sub solutions to an interesting SIR model with incidence rate and delayed diffusion terms. 
\section{Preliminaries}
In this paper we will use some standard notations as $\R^n,\C^n$ for the fields of reals and complex numbers in $n$ dimensions. We also take standard ordering for $\R^3$. This means $u=(u_1,u_2, u_3)^T$ and $v=(v_1,v_2,v_3)^T.$ We say that $u\le v$ if $u_i\le v_i, \ i=1,2,3$ We also say that $u<v$ if $u_i< v_i, \ i=1,2,3$ We also take $|\cdot|$ to be the Euclidean norm in $\R^3.$ 

\medskip
\noindent
The space of all bounded and continuous functions from $\R  \to \R^n$ is denoted by $BC(\R,\R^n)$ which is equipped with the sup-norm $|| f|| := \sup_{t\in\R} \| f(t)\|,$ where $f(t)\in C\left(U, \R^3\right)$, where $U\subset \R$. $BC^k(\R,\R^n)$ stands for the space of all $k$-time continuously differentiable functions $\R\to\R^n$ such that all derivatives up to order $k$ are bounded.

\medskip
\noindent
If the boundedness is dropped from the above function spaces we will simply denote them by $C(\R,\R^n)$ and $C^k(\R,\R^n)$. For $f,g\in BC(\R,\R^n)$ we will use the natural order $f\le g$ if and only if $f(t)\le g(t)$ for all $t\in \R$. A constant function $f(t)=\zeta$ for all $t\in \R$ will be denoted by $\hat \zeta.$ 

\medskip
\noindent
Motivated by the general reaction diffusion system 
\begin{align}\label{RD1}
&\frac{\partial u(x,t)}{\partial t}=D_1\frac{\partial^2 u(x,t-\tau_1)}{\partial x^2} +f_1(u_t(x),v_t(x),w_t(x)),\\ 
&\frac{\partial v(x,t)}{\partial t}=D_2\frac{\partial^2 v(x,t-\tau_2)}{\partial x^2} +f_2(u_t(x),v_t(x),w_t(x))
,\\ 
&\frac{\partial w(x,t)}{\partial t}=D_2\frac{\partial^2 w(x,t-\tau_3)}{\partial x^2} +f_2(u_t(x),v_t(x),w_t(x))
\end{align}
where $t\in\R,  x, u(x,t), v(x,t), w(x,t)\in \R, \ D_i>0, i=1,2,3, \ f_i:C\left([-\tau,0], \R\right)\to \R$ is Lipschitz continuous and $u_t(x) v_t(x), w_t(x)\in C\left([-\tau,0], \R\right),$ defined as 
\[u_t(x)=u(x, t+\theta), v_t(x)=v(x, t+\theta), w_t(x)=w(x, t+\theta) \ \theta\in [-\tau,0], \ t\ge 0, \ x\in \R, \ \tau=\max\{\tau_1, \tau_2, \tau_3\}.\]

\medskip
\noindent
Moreover, we note that $f_{ic}(\phi_{1_{\theta}},\phi_{2_{\theta}}): \X_{c\tau} = C\left([-\tau,0], \R^3\right)\to \R $ is defined by 
\[f_{ic}(\phi,\varphi)= f_i\left(\phi^{1c},\phi^{2c},\phi^{3c}\right), \ \phi^{c_i}(\theta)=\phi_i(c\theta), \ \theta \in [-\tau,0], \ i=1,2,3.\]
We are interested the corresponding general wave system of mixed type
\begin{align}\label{SIR3}
&D_1\phi''(t)-c\phi'(t+r_1)-\beta_1\phi(t+r_1) +f_{1c}(\phi_t,\varphi_t, \psi_t)=0  \\
&D_2\varphi''(t)-c\varphi'(t+r_2)-\beta_2\varphi(t+r_2)+f_{2c}(\phi_t,\varphi_t, \psi_t)=0 \nonumber\\
&D_3\psi''(t)-c\psi'(t+r_3)-\beta_3\psi(t+r_3)+f_{3c}(\phi_t,\varphi_t, \psi_t)=0 \nonumber
\end{align}
\noindent We also assume $\Phi=(\phi_1,\varphi_1, \psi_1), \Psi=(\phi_2,\varphi_2, \psi_2)$ and the following conditions for the nonlinear parts of the system as 
\begin{enumerate}[label=(C\arabic*)]
\item $f(\hat{0})=f(\hat{K})=0,$ where $K=(k_1,k_2, k_3)$ 
\item There are three positive Lipchitz constants $L_1, \ L_2, L_3$ such that 
\begin{align*}
 &\left|f_1(\Phi)-f_1(\Psi)\right|\le L_1 \left|\left|\Phi-\Psi\right|\right|  \\
&\left|f_2(\Phi)-f_2(\Psi)\right|\le L_2 \left|\left|\Phi-\Psi\right|\right|\\
& \left|f_3(\Phi)-f_3(\Psi)\right|\le L_3\left|\left|\Phi-\Psi\right|\right|
\end{align*}
\end{enumerate}
for $\Phi , \Psi \in C\left([-\tau,0], \R\right)),$ and $0\le \phi_i\le M_1, 0\le \varphi_i\le M_2, 0\le \psi_i\le M_3 $ for $i=1,2, $ where \[M_1=\sup_{t\in \R}\phi (t)>k_1, \quad M_2=\sup_{t\in \R}\varphi (t)>k_2,\quad M_3=\sup_{t\in \R}\psi (t)>k_3.\]

\section{Main Results}
\noindent In this section our main results will be discussed. We will start with a brief description a partial quasi monotone condition when there is cross interaction between members of the susceptible, infected and removed populations.  

\subsection{Existence via Schauder's Theorem} 
\noindent We will extend results found in \cite{Barker}, where existence of traveling waves was established for a three vector system using Schauder's fixed point theorem. We will also employ a certain partial quasi monotone condition found in Yang, Liang, Zhang \cite{YangLiang}.    We will first define upper/lower and quasi upper/lower solutions and super/sub solutions. 

\begin{definition} A  pair of functions  \
 $\overline{\Phi}=\left(\overline{\phi}, \overline{\varphi}, \overline{\psi}  \right), \underline{\Phi}=\left(\underline{\phi}, 
  \underline{\varphi}, \underline{\psi}\right)  \in C^2(\R,\R^3),$ where $ \phi, \phi',\phi'',  \varphi, \varphi',\varphi'', \psi, \psi',\psi''$ are bounded on $\R$, is called an  upper solution (lower solution, respectively) for the wave equation (\ref{RD1}) if it satisfies the following
\begin{align*}
&D_1\overline{\phi}''(t)-c\overline{\phi}'(t+r_1)+f_{1c}(\overline{\phi_t}, \overline{\varphi_t}, \overline{\psi_t} )\le 0 , \\
& D_2\overline{\varphi}''(t)-c\overline{\varphi}'(t+r_3)+f_{2c}(\underline{\phi_t}, \overline{\varphi_t}, \underline{\psi_t})\le 0,\\
&D_3\overline{\psi}''(t)-c\overline{\psi}'(t+r_1)+f_{3c}(\overline{\phi_t}, \overline{\varphi_t}, \overline{\psi_t} )\le 0
\end{align*}
and
\begin{align*}
&D_1\underline{\phi}''(t)-c\underline{\phi}'(t+r_1)+f_{1c}(\underline{\phi_t}, \underline{\varphi_t}, \underline{\psi_t} )\ge 0 , \\
& D_2\underline{\varphi}''(t)-c\underline{\varphi}'(t+r_3)+f_{2c}(\overline{\phi_t}, \underline{\varphi_t}, \overline{\psi_t})\le 0,\\
&D_3\underline{\psi}''(t)-c\underline{\psi}'(t+r_1)+f_{3c}(\underline{\phi_t}, \underline{\varphi_t}, \underline{\psi_t} )\le 0
\end{align*}
\end{definition}

\begin{definition} $\overline{\Phi}=\left(\overline{\phi}, \overline{\varphi}, \overline{\psi}  \right), \underline{\Phi}=\left(\underline{\phi}, 
  \underline{\varphi}, \underline{\psi}\right)  \in C^1(\R,\R^3),$ where $ \phi, \phi', \varphi, \varphi',  \psi, \psi'$  are bounded on $\R$, $\phi'' , \varphi'', \psi''$ are locally integrable and essentially bounded on $\R$, is called a quasi-upper solution (quasi-lower solution, respectively) for the wave equation (\ref{RD1}) if it satisfies the following for almost every $t\in \R$
\begin{align*}
&D_1\overline{\phi}''(t)-c\overline{\phi}'(t+r_1)+f_{1c}(\overline{\phi_t}, \overline{\varphi_t}, \overline{\psi_t} )\le 0 , \\
& D_2\overline{\varphi}''(t)-c\overline{\varphi}'(t+r_3)+f_{2c}(\underline{\phi_t}, \overline{\varphi_t}, \underline{\psi_t})\le 0,\\
&D_3\overline{\psi}''(t)-c\overline{\psi}'(t+r_1)+f_{3c}(\overline{\phi_t}, \overline{\varphi_t}, \overline{\psi_t} )\le 0
\end{align*}
and
\begin{align*}
&D_1\underline{\phi}''(t)-c\underline{\phi}'(t+r_1)+f_{1c}(\underline{\phi_t}, \underline{\varphi_t}, \underline{\psi_t} )\ge 0 , \\
& D_2\underline{\varphi}''(t)-c\underline{\varphi}'(t+r_3)+f_{2c}(\overline{\phi_t}, \underline{\varphi_t}, \overline{\psi_t})\le 0,\\
&D_3\underline{\psi}''(t)-c\underline{\psi}'(t+r_1)+f_{3c}(\underline{\phi_t}, \underline{\varphi_t}, \underline{\psi_t} )\le 0
\end{align*}
\end{definition}

\begin{definition} $\overline{\Phi}=\left(\overline{\phi}, \overline{\varphi}, \overline{\psi}  \right), \underline{\Phi}=\left(\underline{\phi}, 
  \underline{\varphi}, \underline{\psi}\right)  \in C(\R,\R^3),$ where $ \phi, \varphi,  \psi$  are bounded on $\R$,  $\phi' , \varphi', \psi' \phi'' , \varphi'', \psi''$ exist almost everywhere, locally integrable and essentially bounded on $\R$, is called a super solution (sub solution, respectively) for the wave equation (\ref{RD1}) if it satisfies the following for almost every $t\in \R$
\begin{align*}
&D_1\overline{\phi}''(t)-c\overline{\phi}'(t+r_1)+f_{1c}(\overline{\phi_t}, \overline{\varphi_t}, \overline{\psi_t} )\le 0 , \\
& D_2\overline{\varphi}''(t)-c\overline{\varphi}'(t+r_3)+f_{2c}(\underline{\phi_t}, \overline{\varphi_t}, \underline{\psi_t})\le 0,\\
&D_3\overline{\psi}''(t)-c\overline{\psi}'(t+r_1)+f_{3c}(\overline{\phi_t}, \overline{\varphi_t}, \overline{\psi_t} )\le 0
\end{align*}
and
\begin{align*}
&D_1\underline{\phi}''(t)-c\underline{\phi}'(t+r_1)+f_{1c}(\underline{\phi_t}, \underline{\varphi_t}, \underline{\psi_t} )\ge 0 , \\
& D_2\underline{\varphi}''(t)-c\underline{\varphi}'(t+r_3)+f_{2c}(\overline{\phi_t}, \underline{\varphi_t}, \overline{\psi_t})\le 0,\\
&D_3\underline{\psi}''(t)-c\underline{\psi}'(t+r_1)+f_{3c}(\underline{\phi_t}, \underline{\varphi_t}, \underline{\psi_t} )\le 0
\end{align*}
\end{definition}

\medskip 
\noindent In the interest of showing the existence of wave front solutions for the system, we fix some constants, $\beta_1,\beta_2, \beta_3>0$ large enough that the following hold. 

\medskip
\noindent 
\begin{definition}{Partial Quasi-Monotone Condition (PQM)}
Fix two constants $\beta_1, \ \beta_2,\beta_3>0$ such that 
\begin{enumerate}[label=(P\arabic*)]
\item $f_{1c}\left(\phi_1,\varphi_1,\psi_1\right)-f_{1c}\left(\phi_2,\varphi_2,\psi_2\right)+\beta_1\left[\phi_1(0)-\phi_2(0)\right]\ge 0,$
\item $f_{2c}\left(\phi_1,\varphi_1,\psi_1\right)-f_{2c}\left(\phi_1,\varphi_2,\psi_1\right)+\beta_2\left[\varphi_1(0)-\varphi_2(0)\right]\ge 0.$
\item $f_{2c}\left(\phi_1,\varphi_1,\psi_1\right)-f_{2c}\left(\phi_2,\varphi_1,\psi_1\right)\le 0,$
\item $f_{2c}\left(\phi_1,\varphi_1,\psi_1\right)-f_{2c}\left(\phi_1,\varphi_1,\psi_2\right)\le 0,$
\item $f_{3c}\left(\phi_1,\varphi_1,\psi_1\right)-f_{3c}\left(\phi_1,\varphi_2,\psi_1\right)+\beta_3\left[\varphi_1(0)-\varphi_2(0)\right]\ge 0.$
\end{enumerate}
where $\phi_1,\ \phi_2, \ \varphi_1,\ \varphi_2, \psi_1,\ \psi_2  \in C([-\tau,0],\R)$,\[ 0\le \phi_2\le \phi_1\le M_1,\ 0\le \varphi_2\le \varphi_1\le M_2, \ 0\le \psi_2\le \psi_1\le M_3.\]
\end{definition}

\medskip
\noindent 
Using Theorem 4.1 in Mallet-Peret, \cite{mal} there is a unique bounded solution for
\begin{equation}\label{fde1}
x''(t)-ax'(t+r)-bx(t+r)=f(t)
\end{equation}
where $a\not=0, b>0$ and $r\in\R$ is assumed to be a small. The characteristic equation of the aforementioned equation is 
\begin{equation}\label{c}
az^2-az e^{rz}-be^{rz}=0.
\end{equation}
 whenever the function $f(t)$ is positive, where $a>0,b>0, r\ne 0$ if the characteristic equation has no zeros on the imaginary axis.
 \begin{proposition}\label{pro 1}
For given $a\not=0,b>0$, and $\lambda_1, \lambda_2$ are the positive and negative roots to characteristic equation for the non-delay differential equation, respectively. The following assertions are true:
\begin{enumerate}
\item For every $r\in\R$, Eq.(\ref{c}) has no root on the imaginary axis;
\item For every $r \le 0$ ($r\ge 0$, respectively), Eq. (\ref{c}) has only one single root in the right half of the complex plane (Eq. (\ref{c}) has only one single root in the left half of the complex plane, respectively), and the root has continuous dependence on $r$.
\item For sufficiently small $r\le 0$ ($r\ge 0$, respectively) there exists only a single root in the strip $\{ z\in\C | \  2\lambda_2 \le \Re z \le 0\}$  (there exists only a single root in the strip $\{ z\in\C | \   0 \le \Re z \le 2\lambda_1\}$, respectively), and the root has continuous dependence on $r$.
\end{enumerate}
\end{proposition}

 \medskip \noindent
The proof relies on Rouche's Theorem and the Implicit Function Theorem. In particular, this was shown in Barker and Nguyen, \cite{BarkNguy}. The solution can be found via an isomorphism,  from $W^{1,p}$ onto $L^p$ using the convolution 
$$({\cal L}^{-1}f)(\xi) = (G*f)(\xi ) =\int^\infty_{-\infty} G(\xi -s)f(s)ds,
$$
where
$$
 W^{1,p}=\{ f\in L^p | \ f \ \mbox{is absolutely continuous, and}\ f'\in L^p\},$$
 and 
 $$ G(\xi) =\frac{1}{2\pi} \int^\infty_{-\infty} e^{i\xi\eta}\Delta (i\eta)^{-1} d\eta 
$$
with the estimate
$$
| G(\xi )| \le Ke^{-\alpha |\xi|}
$$
for some $K>0$ and $\alpha >0$.

\medskip 
\noindent
The importance of quasi-upper/lower solutions are actually needed in order to use the isomorphism from $W^{1,p}$ onto $L^p.$ This is due to the definition requiring that quasi-upper/lower solutions are $C^1$ on the real line and $C^2$ almost everywhere on the real line.

\medskip
\noindent Consider the space \[B_{\mu}(\R,\R^3)=\{\Phi=(\phi, \varphi,\varphi)\in C(\R,\R^3): \ |\Phi|_{\mu}<\infty \},\]
where \[|\Phi|_{\mu}=\sup_{t\in \R} \ e^{-\mu|t|}||\Phi(t)||, \ \mu>0\]
is the exponential decay norm. We are interested in solutions that appear in the closed, bounded, convex subset $\Gamma\subset B_{\mu}$ defined by 
\[
\Gamma=\left\{\left(\phi,\varphi\right)\in C\left(\R,\R^3\right)\middle\vert
\begin{array}{l}\underline{\phi}(t)\le \phi(t)\le \overline{\phi}(t),  \\
\underline{\varphi}(t)\le \varphi(t)\le \overline{\varphi}(t),\\
\underline{\psi}(t)\le \psi(t)\le \overline{\psi}(t)
 \end{array}
\right\}. \] 

\begin{lemma}\label{LVlem1} Define $\Gamma_2\left(\left(\underline{\phi}, \underline{\varphi}, \underline{\psi}\right),\left(\overline{\phi}, \overline{\varphi}, \overline{\psi} \right)\right), |\Phi|_{\mu},$ and $B_{\mu}(\R,\R^3)$ as above, then 
\begin{enumerate}
\item $ \Gamma$ is nonempty.
\item $\Gamma$ is closed, bounded and convex.
\item  $\Gamma\subset B_{\mu}(\R,\R^3).$
\item $\left(B_{\mu}(\R,\R^3),|\cdot|_{\mu}\right)$ is a Banach space.
\end{enumerate}   
\end{lemma}

\noindent Let $\beta_1, \ \beta_2, \beta_3$ be the constants in (PQM), define the following operators 
\begin{align}\label{H}
&H_1(\phi,\varphi, \psi)(t)=f_{1c}\left(\phi_t,\varphi_t, \psi_t\right)+\beta_1\phi(t+r_1), \ \phi, \varphi, \psi \in C(\R,\R,\R),  \\
&H_2(\phi,\varphi, \psi)(t)=f_{2c}\left(\phi_t,\varphi_t, \psi_t\right)+\beta_2\varphi(t+r_2), \ \phi, \varphi, \psi \in C(\R,\R, \R),\\
&H_3(\phi,\varphi, \psi)(t)=f_{3c}\left(\phi_t,\varphi_t, \psi_t\right)+\beta_3\psi(t+r_3), \ \phi, \varphi, \psi \in C(\R,\R, \R). 
\end{align}

\medskip \noindent We have the following Lemma

\begin{lemma}\label{lemH}
Assume (C1) and (PQM) hold, $(0,0,0)\le (\phi, \varphi, \psi)\le (M_1,M_2,M_3)$ and $0\le \phi_2(t)\le \phi_1(t)\le M_1, 0\le \varphi_2(t)\le \varphi_1(t)\le M_2, 0\le \psi_2(t)\le \psi_1(t)\le M_3 \ \forall t \in \R  $ then the operators $H_1, H_2, H_3$ satisfy
\begin{enumerate}
    \item $H_1(\phi, \varphi, \psi)(t)\ge 0, H_2(\phi, \varphi, \psi)(t)\ge 0, H_3(\phi, \varphi, \psi)(t)\ge 0,$
    \item $H_1(\phi_2, \varphi_2, \psi_2)(t)\le H_1(\phi_1, \varphi_1, \psi_1)(t),$
    \item $H_3(\phi_2, \varphi_2, \psi_2)(t)\le H_3(\phi_1, \varphi_1, \psi_1)(t),$
    \item $H_2(\phi_1, \varphi_1, \psi_1)(t)\le H_2(\phi_2, \varphi_1, \psi_1)(t),$
     \item $H_2(\phi_1, \varphi_2, \psi_1)(t)\le H_2(\phi_1, \varphi_1, \psi_1)(t),$
     \item $H_2(\phi_1, \varphi_1, \psi_1)(t)\le H_2(\phi_1, \varphi_1, \psi_2)(t),$
\end{enumerate}
\end{lemma}
\begin{proof}
 The proof is a direct calculation using $(C1)$ and $(PQM).$   
\end{proof}
\medskip
\noindent Clearly, the following system is equivalent to system (\ref{RD1})
\begin{align}\label{SIR4}
&D_1\phi''(t)-c\phi'(t+r_1)-\beta_1\phi(t+r_1) +H_1\left(\phi,\varphi, \psi\right)(t)=0  \\
&D_2\varphi''(t)-c\varphi'(t+r_2)-\beta_2\varphi(t+r_2)+H_2\left(\phi,\varphi, \psi\right)(t)=0, \nonumber\\
&D_3\psi''(t)-c\psi'(t+r_3)-\beta_3\psi(t+r_3)+H_3\left(\phi,\varphi, \psi\right)(t)=0
\end{align}

\medskip
\noindent
The characteristic equations of system are 
\begin{align}\label{ce}
& D_1\eta^2-c\eta e^{\eta r_1}-\beta_1 e^{\eta r_1}\\
& D_2\eta^2-c\eta e^{\eta r_2}-\beta_2 e^{\eta r_2} \nonumber\\
& D_3\eta^2-c\eta e^{\eta r_3}-\beta_3 e^{\eta r_3}
\end{align}

\medskip
\noindent
Since, $c\neq 0, \beta_i>0, i=1,2,3$ by  proposition \ref{pro 1}, there are no roots on the imaginary axis.  Using the isomorphism detailed above the following operator is equivalent to system \ref{SIR4}:
\begin{equation}\label{fixed point1}
( \phi, \varphi,\psi)^T = -{\cal L}^{-1}H(\phi,\varphi, \psi)=\left(-{\cal L}_1^{-1}H_1(\phi,\varphi,\psi),-{\cal L}_2^{-1}H_2(\phi,\varphi,\psi), ,-{\cal L}_3^{-1}H_3(\phi,\varphi,\psi)\right)^T , \ \phi,\varphi, \psi \in \Gamma.
\end{equation}
 \noindent We label ${\cal L}^{-1}H(\Phi(t))=F(\phi,\varphi,\psi)(t)$ giving the operator, 

\begin{equation}\label{Solution}
 F\left(\phi,\varphi,\psi\right)(t)=\int^\infty_{-\infty} G(t-s,r)H(\phi,\varphi,\psi)(s)ds.   
\end{equation}
In fact, we are interested in the iterative operator
\begin{equation}\label{Itop}
  F^n\left(\phi,\varphi,\psi\right)(t)=\int^\infty_{-\infty} G(t-s,r)H_{n-1}(\phi,\varphi,\psi)(s)ds,
 \end{equation}
where $H_0=(\overline{\phi},\overline{\varphi}, \overline{\psi}),(\underline{\phi},\underline{\varphi}, \underline{\psi})$ are super/sub solutions.
\begin{lemma}\label{lemF}
Assume $(C1)$ and $(PQM)$ hold, $(0,0,0)\le (\phi, \varphi, \psi)\le (M_1,M_2,M_3)$ and $0\le \phi_2(t)\le \phi_1(t)\le M_1, 0\le \varphi_2(t)\le \varphi_1(t)\le M_2, 0\le \psi_2(t)\le \psi_1(t)\le M_3 \ \forall t \in \R  $ then the operators $F_1, F_2, F_3$ satisfy
\begin{enumerate}
    \item $F_1(\phi, \varphi, \psi)(t)\ge 0, F_2(\phi, \varphi, \psi)(t)\ge 0, F_3(\phi, \varphi, \psi)(t)\ge 0,$
    \item $F_1(\phi_2, \varphi_2, \psi_2)(t)\le F_1(\phi_1, \varphi_1, \psi_1)(t),$
    \item $F_3(\phi_2, \varphi_2, \psi_2)(t)\le F_3(\phi_1, \varphi_1, \psi_1)(t),$
    \item $F_2(\phi_1, \varphi_1, \psi_1)(t)\le F_2(\phi_2, \varphi_1, \psi_1)(t),$
     \item $F_2(\phi_1, \varphi_2, \psi_1)(t)\le F_2(\phi_1, \varphi_1, \psi_1)(t),$
     \item $F_2(\phi_1, \varphi_1, \psi_1)(t)\le F_2(\phi_1, \varphi_1, \psi_2)(t),$
\end{enumerate}
\end{lemma}

\medskip
\noindent
The proof follows from the inherited properties of $\left(H_1\left(\phi,\varphi,\psi\right),H_2\left(\phi,\varphi,\psi\right)\right)$ from Lemma (\ref{lemH}). 

\noindent We have the following results.
\begin{lemma}  Consider the operator $F$ and the ball $B_{\mu},$ then
\begin{enumerate}
\item $F(\Gamma)\subset \Gamma.$
\item $F$ is continuous with respect to  $|\cdot|_{\mu}$ in $B_{\mu}(\R,\R^3).$
\item $F(\Gamma)\to \Gamma $ is compact.
\end{enumerate}
\end{lemma}
\begin{proof}
For part $i.)$ we only need to show that 
\begin{align*}
\begin{cases}
\underline{\phi}\le F_1(\underline{\phi},\underline{\varphi}, \underline{\psi})\le F_1(\overline{\phi},\overline{\varphi},\overline{\psi} )   \le\overline{\phi} \\
\underline{\varphi}\le  F_2(\overline{\phi},\underline{\varphi}, \overline{\psi}) \le F_2(\underline{\phi},\overline{\varphi}, \underline{\psi})    \le \overline{\varphi}.
\\
\underline{\psi}\le F_3(\underline{\phi},\underline{\varphi}, \underline{\psi})\le F_3(\overline{\phi},\underline{\varphi}, \overline{\psi})   \le\overline{\psi} 
\end{cases}  
\end{align*}
This is due to the fact 
\begin{align*}
\begin{cases}
 F_1(\underline{\phi},\underline{\varphi}, \underline{\psi})\le F_1(\phi,\varphi,\psi)\le  F_1(\overline{\phi},\overline{\varphi},\overline{\psi} )     \\
 F_2(\overline{\phi},\underline{\varphi}, \overline{\psi}) \le F_2(\phi,\varphi,\psi)\le F_2(\underline{\phi},\overline{\varphi}, \underline{\psi})\\
 F_3(\underline{\phi},\underline{\varphi}, \underline{\psi})\le F_3(\phi,\varphi,\psi)\le  F_3(\overline{\phi},\overline{\varphi},\overline{\psi} ).
\end{cases}  
\end{align*}
We only show the first inequality, because the other two follows the same way. We first show that $ \underline{\phi}\le F_1(\underline{\phi},\underline{\varphi}, \underline{\psi})$. 
\begin{align*}
&F_1(\underline{\phi},\underline{\varphi}, \underline{\psi})(t)=\int_{-\infty}^{\infty} G_1(t-s,r)H_1 (\underline{\phi},\underline{\varphi}, \underline{\psi})(s) ds\\
&\ge \int_{-\infty}^{\infty} G_1(t-s,r)\left(-D_1\underline{\phi}''(t)+c\underline{\phi}'(t+r_1)+\beta_1\underline{\phi}(t+r_1)\right) ds=\underline{\phi}(t)
\end{align*}
for all $t\in \R.$
Next, we show $F_1(\overline{\phi},\overline{\varphi},\overline{\psi} )   \le\overline{\phi}$.
\begin{align*}
&F_1(\overline{\phi},\overline{\varphi},\overline{\psi} )(t)=\int_{-\infty}^{\infty} G_1(t-s,r)H_1 (\overline{\phi},\overline{\varphi},\overline{\psi} )(s) ds\\
&\le \int_{-\infty}^{\infty} G_1(t-s,r)\left(-D_1\overline{\phi}''(t)+c\overline{\phi}'(t+r_1)+\beta_1\overline{\phi}(t+r_1)\right) ds=\overline{\phi}(t)
\end{align*}
for all $t\in \R.$

\medskip \noindent
For part $ii.)$ we only need to show 
\[F_1:B_{\mu}(\R,\R^3)\to B_{\mu}(\R,\R^3) \] is continuous with respect to $|\cdot|_{\mu},$ because the proof for $F_2, F_3$ is very similar. Take $\mu<\max\{\delta_1,\delta_2,\delta_3\}$ and let $\Phi, \Psi \in B_{\mu}(\R,\R^3). $ It is clear that \[F_1=(F_1(\Phi), F_1(\Psi)) \in B_{\mu}(\R,\R^2). \] 
We will now turn our attention to the continuity of $F$. Fix $\varepsilon>0,$ and take \[\delta<\min\left\{\frac{\varepsilon(\delta_1-\mu)}{2 K_1},\frac{\varepsilon}{e^{\mu c \tau}L_1+\beta_1}\right\},\]
where $L_1$ is the Lipchitz constant from (C2), $\beta_1$ is from (PQM), and $\delta_1, K_1$ are from the bounds of the Green function. We will first prove that 
$H_1:B_{\mu}(\R,\R^3)\to B_{\mu}(\R,\R^3) $ is continuous with respect to $|\cdot|_{\mu}$. We take $\left|\Phi(t)-\Psi(t)\right|_{\mu}<\delta,$ then 
\begin{align*}
&A:=\left|H_1(\Phi)(t)-H_1(\Psi)(t)\right|_{\mu}\\
&=\left|f_{1c}(\phi_{1t},\varphi_{1t}, \psi_{1t})+\beta_1\phi_1(t+r_1)-\left(f_{2c}(\phi_{2t},\varphi_{2t}, \psi_{2t})+\beta_1\phi_2(t+r_1)\right) \right|_{\mu}\\
&\le \left|f_{1c}(\phi_{1t},\varphi_{1t}, \psi_{1t})-f_{2c}(\phi_{2t},\varphi_{2t}, \psi_{2t})\right|_{\mu}+\beta_1\left|\phi_1(t+r_1)-\phi_2(t+r_1) \right|_{\mu}\\
&\le L_1 ||\Phi(t)-\Psi(t)||_{\X_{c\tau}}e^{-\mu|t|} +\beta_1 \sup_{t\in \R}\left|\phi_1(t)-\phi_2(t) \right|_{\mu}\\
&\le L_1 \sup_{\theta\in (-c\tau,0)} |\Phi(t+\theta)-\Psi(t+\theta)|e^{-\mu|t+\theta|}+\beta_1 |\Phi(t)-\Psi(t)|_{\mu}\\
&\le e^{\mu c\tau} L_1  |\Phi(t)-\Psi(t)|_{\mu}+\beta_1 \left|\Phi(t)-\Psi(t) \right|_{\mu}<\left(e^{\mu c\tau} L_1+\beta_1\right)\delta\\
&<\left(e^{\mu c\tau} L_1+\beta_1\right)\left(\frac{\varepsilon}{e^{\mu c \tau}L_1+\beta_1}\right)<\varepsilon.
\end{align*}
Thus, $H_1:B_{\mu}(\R,\R^3)\to B_{\mu}(\R,\R^3)$ is continuous. The proof for $H_2, H_3$ are similar. Therefore, $H=(H_1,H_2):B_{\mu}(\R,\R^3)\to B_{\mu}(\R,\R^3)$ is continuous with respect to $|\cdot|_{\mu}.$ 
We can now prove that $F_1$ is continuous in the same manner. Indeed,
\begin{align*}
&B:=\left|F_1(\Phi)(t)-F_1(\Psi)(t)\right|=\left|F_1(\phi_{1},\varphi_{1},\psi_{1})-F_1(\phi_{2},\varphi_{2},\psi_{2})\right| \\
&=\left|\int_{-\infty}^{\infty} G_1(t-s,r)\left(H_1(\phi_{1},\varphi_{1},\psi_{1})(s)-H_1(\phi_{2},\varphi_{2},\psi_{2})(s)\right) ds \right| \\
&\le \int_{-\infty}^{t}\left| G_1(t-s,r)\left(H_1(\phi_{1},\varphi_{1},\psi_{1})(s)-H_1(\phi_{2},\varphi_{2},\psi_{2})(s)\right)\right| ds \\
&+\int_{t}^{\infty} \left| G_1(t-s,r)\left(H_1(\phi_{1},\varphi_{1},\psi_{1})(s)-H_1(\phi_{2},\varphi_{2},\psi_{2})(s)\right)\right| ds \\
&=\int_{-\infty}^{t}\left| G_1(t-s,r)\right|\left|H_1(\phi_{1},\varphi_{1},\psi_{1})(s)-H_1(\phi_{2},\varphi_{2},\psi_{2})(s)\right|ds\\
&+\int_{t}^{\infty}\left| G_1(t-s,r)\right|\left|H_1(\phi_{1},\varphi_{1},\psi_{1})(s)-H_1(\phi_{2},\varphi_{2},\psi_{2})(s)\right|ds.
\end{align*}

\noindent
We can now use the fact that $H_1(\phi,\varphi)(t)$ is continuous with respect to the exponential decay norm.
\begin{align*}
&B=\int_{-\infty}^{t}\left| G_1(t-s,r)\right|e^{\mu |s|}\left|H_1(\phi_{1},\varphi_{1},\psi_{1})(s)-H_1(\phi_{2},\varphi_{2},\psi_{2})(s)\right|e^{-\mu |s|}ds\\
&+\int_{t}^{\infty}\left| G_1(t-s,r)\right|e^{\mu |s|}\left|H_1(\phi_{1},\varphi_{1},\psi_{1})(s)-H_1(\phi_{2},\varphi_{2},\psi_{2})(s)\right|e^{-\mu |s|}ds \\
&\le\left(\int_{-\infty}^{t}\left| G_1(t-s,r)\right|e^{\mu |s|}ds+\int_{t}^{\infty}\left| G_1(t-s,r)\right|e^{\mu |s|}ds\right)\\
&\times \sup_{s\in \R}\left|H_1(\phi_{1},\varphi_{1},\psi_{1})(s)-H_1(\phi_{2},\varphi_{2},\psi_{2})(s)\right|_{\mu}. 
\end{align*}
\noindent
 Here, we will use the fact that $G_1(t,r)$ is bounded and decays as $|t|\to \infty$ to see the following:
\begin{align*}
&B\le \delta\left(\int_{-\infty}^{t}\left| G_1(t-s,r)\right|e^{\mu |s|}ds+\int_{t}^{\infty}\left| G_1(t-s,r)\right|e^{\mu |s|}ds\right)\\
&\le  \delta\left(\int_{-\infty}^{t}M_1e^{-\delta_1 |t-s|}e^{\mu|s|}ds+\int_{t}^{\infty}M_1e^{-\delta_1 |t-s|}e^{\mu|s|}ds\right)\\
&\le  M_1\delta\left(\int_{0}^{t}e^{-\delta_1 (t-s)+\mu s}ds+\int_{-\infty}^{0}e^{-\delta_1 (t-s)-\mu s}ds+\int_{t}^{\infty}e^{\delta_1 (t-s)+\mu s}ds\right)<2K_1\delta\left(\frac{\delta_1 e^{\mu t}+\mu e^{-\delta_1 t}}{\delta_1^2-\mu^2}\right).
\end{align*}
\noindent
In the exponential decay norm for $t>0$ we see 
\begin{align*}
&\left|F_1(\Phi)(t)-F_1(\Psi)(t)\right|_{\mu}=\left|F_1(\Phi)(t)-F_1(\Psi)(t)\right|e^{-\mu t}\\
&\le 2K_1\delta \left(\frac{\delta_1 e^{\mu t}+\mu e^{-\delta_1 t}}{\delta_1^2-\mu^2}\right)e^{-\mu t}=2K_1\delta \left(\frac{\delta_1 +\mu e^{-(\delta_1+\mu)t}}{\delta_1^2-\mu^2}\right)\\
&\le 2K_1\delta \left(\frac{\delta_1 +\mu}{\delta_1^2-\mu^2}\right)= 2K_1\delta \left(\frac{1}{\delta_1-\mu}\right)<\varepsilon
\end{align*}
Moreover, when $t\le 0$ we have the following estimate 
\begin{align*}
&\left|F_1(\Phi)(t)-F_1(\Psi)(t)\right|=\left|F_1(\phi_{1},\varphi_{1},\psi_{1})-F_1(\phi_{2},\varphi_{2},\psi_2)\right| \\ 
&\le\left(\int_{-\infty}^{t}\left| G_1(t-s,r)\right|e^{\mu |s|}ds+\int_{t}^{\infty}\left| G_1(t-s,r)\right|e^{\mu |s|}ds\right)\\
&\times \sup_{s\in \R}\left|H_1(\phi_{1},\varphi_{1},\psi_{1})(s)-H_1(\phi_{2},\varphi_{2},\psi_{2})\right|_{\mu} \\
&\le \delta\left(\int_{-\infty}^{t}\left| G_1(t-s,r)\right|e^{\mu |s|}ds+\int_{t}^{\infty}\left| G_1(t-s,r)\right|e^{\mu |s|}ds\right)\\
&\le  K_1\delta\left(\int_{t}^{0}e^{\delta_1 (t-s)-\mu s}ds+\int_{-\infty}^{t}e^{-\delta_1 (t-s)-\mu s}ds+\int_{0}^{\infty}e^{\delta_1 (t-s)+\mu s}ds\right)\\
&<2K_1\delta\left(\frac{\delta e^{-\mu t}+\mu e^{\delta_1 t}}{\delta_1^2-\mu^2}\right).
\end{align*}

\medskip
\noindent
Similarly as above in the exponential norm for $t\le0$ we see 
\begin{align*}
&\left|F_1(\Phi)(t)-F_1(\Psi)(t)\right|_{\mu}=\left|F_1(\Phi)(t)-F_1(\Psi)(t)\right|e^{\mu t}\\
&\le 2K_1\delta \left(\frac{\delta_1 e^{-\mu t}+\mu e^{\delta_1 t}}{\delta_1^2-\mu^2}\right)e^{\mu t}=2M_1\delta \left(\frac{\delta_1 +\mu e^{(\delta_1+\mu)t}}{\delta_1^2-\mu^2}\right)\\
&\le 2K_1\delta \left(\frac{\delta_1 +\mu}{\delta_1^2-\mu^2}\right)= 2M_1\delta \left(\frac{1}{\delta_1-\mu}\right)<\varepsilon.
\end{align*}
Therefore, 
\[F=(F_1,F_2, F_3):B_{\mu}(\R,\R^3)\to B_{\mu}(\R,\R^3) \] is continuous with respect to $|\cdot|_{\mu}.$

\medskip \noindent For part $iii.)$ we note that using parts $i.)$ and $ii.)$ show that $F^n(\Gamma)$ is equicontinuous and uniformly bounded on any finite interval in $\R$. Thus, take $n\in \N$, then in the interval $[-n,n]$ we can say that $F^n$ is compact by Arzela-Ascoli. Now,  define 
\[F(\Phi,\Psi)(t)=\begin{cases} F(\Phi,\Psi)(t), & \ \text{when}\ t\in [-n,n]\\
F(\Phi,\Psi)(n), & \ \text{when}\ t\in (n,\infty)\\
F(\Phi,\Psi)(-n), & \ \text{when}\ t\in (-\infty,-n).
\end{cases}\]
\\ Fix $t\in \R,$ then for all  $(\phi(t),\varphi(t),\psi)\in \Gamma. $
\begin{align*}
&\sup_{t\in \R} |F^n(\Phi,\Psi)(t)-F(\Phi,\Psi)(t)|_{\mu}\\
&=\sup_{t\in (-\infty,-n)\bigcup (n,\infty)}|F^n(\Phi,\Psi)(t)-F(\Phi,\Psi)(t)|_{\mu}\\
&\le 2C_1 e^{-\mu n}
\end{align*}
Thus, $\lim_{n\to \infty} 2C_1 e^{-\mu n}=0,$ 
 so $F^n\to F$ in $\Gamma$ as $n\to \infty.$ 
We can apply Arzela-Ascoli to $F$ as well. The proof is complete.
\end{proof}
\begin{theorem}
 Assume $ \mathcal{R}_0>1, c\ge c^*,$ if there is an upper $(\overline{\phi},\overline{\varphi}, \overline{\psi})$ and a lower solution $(\underline{\phi},\underline{\varphi}, \underline{\psi}),$ both in $\Gamma$ of Eq.(\ref{RD1}) such that for all $t\in \R$
\[ 0\le  \underline{\phi}(t)\le \phi \le \overline{\phi}(t), \ 0\le  \underline{\varphi}(t)\le \varphi \le \overline{\varphi}(t),  0\le  \underline{\psi}(t)\le \phi \le \overline{\psi}(t)  .\]
Then, there exists a monotone traveling wave solution to the system (\ref{RD1}).
\end{theorem} 

\noindent It is extremely difficult to find $C^2$ solutions, so we need the following propositions, which allows us to find super/sub solutions, instead.
\begin{proposition}
Let $(\phi,\varphi,\psi)$ be a quasi- upper solution (quasi-lower solution, respectively) of Eq. (\ref{RD1}). Then, $F(\phi,\varphi,\psi)$ is an upper solution (lower solution, respectively) of Eq. (\ref{RD1}).
\end{proposition}

\begin{proposition}
Let $(\phi,\varphi, \psi)$ be a super solution (sub) solution, respectively) of Eq. (\ref{RD1}). Then, $F(\phi,\varphi, \psi)$ is an quasi-upper solution (qusi-lower solution, respectively) of Eq. (\ref{RD1}).
\end{proposition}
\noindent This leads to the following corollary.
\begin{corollary}
Assume $ \mathcal{R}_0>1, c\ge c^*,$ and there is a super, sub solution, respectively  $(\overline{\phi}, \overline{\varphi}, \overline{\psi}),$ $(\underline{\phi},\underline{\varphi}, \underline{\psi}),$  of Eq. (\ref{RD1}). Then, there exists a monotone traveling wave solution of the system (\ref{RD1}).   \end{corollary}
\section{Applications}
\noindent Consider the closed SIR model with bi-linear incidence
\begin{align}\label{SIR6}
&\frac{\partial S(x,t)}{\partial t}=D_S\frac{\partial^2 S(x,t-\tau_1)}{\partial x^2}+B-\mu_1S(x,t)-\frac{\beta S(x,t)I(x, t-\tau_4)}{1+\alpha I(x, t-\tau_4) }, \\
&\frac{\partial I(x,t)}{\partial t}=D_I\frac{\partial^2 I(x,t-\tau_2)}{\partial x^2}+\frac{\beta S(x,t)I(x, t-\tau_4)}{1+\alpha I(x, t-\tau_4) }-(\mu_2+\gamma)I(x,t) \nonumber\\
&\frac{\partial R(x,t)}{\partial t}=D_R\frac{\partial^2 R(x,t-\tau_3)}{\partial x^2}+\gamma I(x,t) -\mu_3 R(x,t).\nonumber
\end{align}
\noindent Taking $D_s=D_I=D_R=1$ for brevity. The endemic steady state is 
\[S^*=\frac{B\alpha +\mu_2+\gamma}{\beta+\alpha\mu_1}, \quad I^*=\frac{B\alpha-\mu_1(\mu_2+\gamma)}{(\beta+\alpha\mu_1)(\mu_2+\gamma)}, \quad R^*=\frac{\gamma[B\alpha-\mu_1(\mu_2+\gamma)]}{\mu_3(\beta+\alpha\mu_1)(\mu_2+\gamma)}.\]
and defining  $N=I+S+R$ we see the system (\ref{SIR6}) gives
\begin{align}\label{SIR7}
&\frac{\partial N(x,t)}{\partial t}=\frac{\partial^2 N(x,t-\tau_1)}{\partial x^2}+B-\mu_1 N(x,t)-(\mu_2-\mu_1)I(x,t)-(\mu_3-\mu_1)R(x,t), \\
&\frac{\partial I(x,t)}{\partial t}=\frac{\partial^2 I(x,t-\tau_2)}{\partial x^2}+\frac{\beta \left(N(x,t)-I(x,t)-R(x,t)\right)I(x, t-\tau_4)}{1+\alpha I(x, t-\tau_4) }-(\mu_2+\gamma)I(x,t) \nonumber\\
&\frac{\partial R(x,t)}{\partial t}=\frac{\partial^2 R(x,t-\tau_3)}{\partial x^2}+\gamma I(x,t) -\mu_3 R(x,t).\nonumber
\end{align}

\medskip \noindent Making the invariant change of variables and dropping the tildes  $\tilde{N}=B/\mu_1-N, \tilde{I}=I, \tilde{R}=R $ gives
\begin{align}\label{SIR8}
&\frac{\partial N(x,t)}{\partial t}=\frac{\partial^2 N(x,t-\tau_1)}{\partial x^2}-\mu_1 N(x,t)+(\mu_2-\mu_1)I(x,t)+(\mu_3-\mu_1)R(x,t), \\
&\frac{\partial I(x,t)}{\partial t}=\frac{\partial^2 I(x,t-\tau_2)}{\partial x^2}+\frac{\beta \left(B/\mu_1-I(x,t)-R(x,t)\right)I(x, t-\tau_4)}{1+\alpha I(x, t-\tau_4) }-(\mu_2+\gamma)I(x,t) \nonumber\\
&\frac{\partial R(x,t)}{\partial t}=\frac{\partial^2 R(x,t-\tau_3)}{\partial x^2}+\gamma I(x,t) -\mu_3 R(x,t).\nonumber
\end{align}

 \noindent Using the wave transformation as the following \[S(x,t)=\phi(x+ct), \ I(x,t)=\varphi(x+ct), \ R(x,t)=\psi(x+ct)\] where $ c>0$ is the wave speed. Applying the wave transformation and letting $\xi=x+ct$ and moving the delay out of the diffusion term gives the equivalent system using $t$ as the variable.

\begin{align}\label{wave 1}
&{\phi}''(t)-c{\phi}'(t+r_1)-\mu_1 \phi(t+r_1)+(\mu_2-\mu_1)\varphi(t+r_1)+(\mu_3-\mu_1)\psi(t+r_1)=0 \\
& {\varphi}''(t)-c{\varphi}'(t+r_2)-(\mu_2+\gamma)\varphi(t+r_2)\\
&-\frac{\beta \left(B/\mu_1-\phi(t+r_2)-\varphi(t+r_2)-\psi(t+r_2)\right)\varphi(t+r_2-r_4)}{1+\alpha \varphi(t+r_2-r_4) }=0\nonumber,\\
&{\psi}''(t)-c{\psi}'(t+r_3)-\mu_3 \psi(t+r_3)+\gamma\psi(t+r_3)=0
\end{align}

\medskip
\noindent It is easy to see that system (\ref{wave 1}) satisfies the asymptotic conditions
\[\lim_{t\to -\infty} (\phi(t),\varphi(t),\psi(t))=(0,0,0), \quad \lim_{t\to \infty} (\phi(t),\varphi(t),\psi(t))=(k_1,k_2,k_3).\]
\begin{lemma}
The nonlinear terms, $f_{ci}, i=1,2,3$ of system (\ref{SIR3}) satisfy (PQM).      
\end{lemma}
\begin{proof}
The nonlinear terms are
\begin{align*}
&f_{c1}(\phi,\varphi,\psi)(t)=-\mu_1\phi(0)+(\mu_2-\mu_1)\varphi(0)+(\mu_3-\mu_1)\psi(0)  \\
&f_{c2}(\phi,\varphi)(t)=-(\mu_2+\gamma)\varphi(0)-\frac{\beta \left(B/\mu_1-\phi(0)-\varphi(0)-\psi(0)\right)\varphi(-r_4)}{1+\alpha \varphi(-r_4) }\\
&f_{c3}(\phi,\varphi)(t) =-\mu_3 \psi(t+r_3)+\gamma\psi(t+r_3).
\end{align*}
Let $0\le\phi_2(s)\le\phi_1(s)\le M_1, \ 0\le\varphi_2(s)\le\varphi_1(s)\le M_2, \  \ 0\le\psi_2(s)\le\psi_1(s)\le M_2, $ where $\phi_i, \, \varphi_i, \psi_i \in C([-\tau_,0],\R).$ 

\noindent
For (P1),
\begin{align*}
&f_{c1}(\phi_1,\varphi_1,\psi_1)(t)-f_{c1}(\phi_2,\varphi_2,\psi_2)(t)\\
&= -\mu_1\phi_1(0)+(\mu_2-\mu_1)\varphi_1(0)+(\mu_3-\mu_1)\psi_1(0)-\left[-\mu_1\phi_2(0)+(\mu_2-\mu_1)\varphi_2(0)+(\mu_3-\mu_1)\psi_2(0)\right]\\
&=-\mu_1\left(\phi_1(0)-\phi_2(0)\right)+(\mu_2-\mu_1)\left(\varphi_1(0)-\varphi_2(0)\right)+(\mu_3-\mu_1)\left(\psi_1(0)-\psi_2(0)\right)\\
&\ge -\mu_1\left(\phi_1(0)-\phi_2(0)\right). 
\end{align*} Define $\beta_1\ge \mu + \beta \varphi_1(-r_2), $ then $f_{1c}\left(\phi_2,\varphi_1\right)+\beta_1\left[\phi_1(0)-\phi_2(0)\right]\ge 0.$
Take $\beta_1=\mu_1$ and you have the result.

\medskip \noindent The rest of the conditions can be shown in a similar manner to Lemma 0.7 in \cite{YangLiang}.
\end{proof}

\medskip \noindent
Consider the quadratic equations 
\begin{align*}
&P_1(\lambda)=\lambda^2-c\lambda -\mu_1+(\mu_2-\mu_1)\frac{M_2}{M_1} +(\mu_3-\mu_1)\frac{M_3}{M_1} \\
&P_2(\lambda)=\lambda^2-c\lambda -\mu_1 \\
&P_3(\lambda)=\lambda^2-c\lambda +\beta B M_2\mu_1-(\gamma +\mu_2)  \\
&P_4(\lambda)=\lambda^2-c\lambda-(\gamma +\mu_2)\\
&P_5(\lambda)=\lambda^2-c\lambda+\gamma M_2/M_3 -\mu_3\\
&P_6(\lambda)=\lambda^2-c\lambda -\mu_3. 
\end{align*}
Each of the equations has at least one positive root. Here we take the smallest positive root of each equation and label them in accordance to their subscript $\lambda_1, i=1-6.$ 
Define the following exponential polynomials 
\begin{align*}
&\Delta_1(\eta)=\eta^2-c\eta e^{r_1\eta} -\mu_1e^{r_1\eta}+(\mu_2-\mu_1)\frac{M_2}{M_1}e^{r_1\eta} +(\mu_3-\mu_1)\frac{M_3}{M_1}e^{r_1\eta}  \\
&\Delta_2(\eta)=\eta^2-c\eta e^{r_1\eta}-\mu_1 e^{r_1\eta}\\
&\Delta_3(\eta)=\eta^2-c\eta e^{r_2\eta}+\beta B M_2e^{r_2\eta}-(\gamma +\mu_2)e^{r_2\eta}\\
&\Delta_4(\eta)=\eta^2-c\eta e^{r_2\eta}-(\gamma +\mu_2)e^{r_2\eta}\\
&\Delta_5(\eta)=\eta^2-c\eta e^{r_3\eta}+\gamma M_2/M_3 -\mu_3e^{r_3\eta}\\
&\Delta_6(\eta)=\eta^2-c\eta e^{r_3\eta}-\mu_3e^{r_3\eta}. 
\end{align*}

\begin{lemma}
Let $c>c^*$ and consider the equations $\Delta_i(\eta), i=1-6.$ Each $\Delta_i(\eta)$  
has a positive root continuously dependent on some small $r>0$  respectively denoted as $\eta_i(r).$  Moreover, the roots are  real and
\begin{equation}
\lim_{r\to 0} \eta_i(r) =\lambda_i.
\end{equation}
\end{lemma}

\begin{proof} The proof is an application of the implicit function theorem. For a detailed proof the reader is referred to \cite{Bark3}. 
\end{proof}
\medskip \noindent
There exists  $\varepsilon_i>0, i=1-6$ , such that for some small $\varepsilon_0$

\begin{align}\label{solcond}
& \mu_1(k_1+\varepsilon_1)-k_1\eta_1^2-(\mu_2-\mu_1)M_2 -(\mu_3-\mu_1)M_3 >\varepsilon_0 \\
&\mu_1(k_1-\varepsilon_2)-(\mu_2-\mu_1)(k_2-\varepsilon_4)-(\mu_3-\mu_1)(k_3-\varepsilon_6)>\varepsilon_0  \nonumber\\
&(\mu_2+\gamma)(k_2+\varepsilon_3)-k_2\eta_2^2 -\frac{\beta \left(B/\mu_1-M_1-k_2+\varepsilon_3-M_3\right)M_2}{1+\alpha M_2}>\varepsilon_0\nonumber\\
& \frac{\beta \left(B/\mu_1-M_1-k_2+\varepsilon_4-M_3\right)(k_2-\varepsilon_4)}{1+\alpha (k_2-\varepsilon_4)}-(\mu_2+\gamma)(k_2-\varepsilon_4)>\varepsilon_0\nonumber\\
&-\eta^2_5k_3-\mu_3(k_3+\varepsilon_5)+\gamma M_2>\varepsilon_0 \nonumber\\
&-\mu_3(k_3-\varepsilon_6)+\gamma (k_2-\varepsilon_4)>\varepsilon_0 \nonumber. . 
\end{align}

\medskip \noindent It is now possible to find $t_i, i=1,\cdots 6$ such that $t_1<t_2t_5<t_6 <t_3<t_4$, such that for some $\eta>0$ the following functions are continuous everywhere and continuously differentiable almost everywhere.
\begin{align*}
&\overline{\phi}(t)=\begin{cases}  k_1e^{\eta_1t} \quad t\le t_1
\\
k_1+\varepsilon_1 e^{-\eta t} \quad t >t_1. \end{cases} \quad  \underline{\phi}(t)=\begin{cases}  \frac{k_1e^{\eta_2t}}{M} \quad t\le t_4\\
k_1-\varepsilon_2 e^{-\eta t} \quad t> t_4 ,
\end{cases}  \\
&\overline{\varphi}(t)=\begin{cases}  k_2e^{\eta_3t}, \ t\le t_2 ,\\
k_2+\varepsilon_3 e^{-\eta t}, \ 0>  t_2  
\end{cases} 
  \quad  \underline{\varphi}(t)=\begin{cases}  \frac{k_2e^{\eta_4t}}{M} \quad t\le t_5\\
k_2-\varepsilon_4 e^{-\eta t} \quad t> t_5.
\end{cases}    \\
&\overline{\psi}(t)=\begin{cases}  k_3e^{\eta_5t}, \ t\le t_3 ,\\
k_3+\varepsilon_5 e^{-\eta t}, \ 0>  t_3  
\end{cases} 
  \quad  \underline{\psi}(t)=\begin{cases}  \frac{k_3e^{\eta_6t}}{M} \quad t\le t_6\\
k_3-\varepsilon_6 e^{-\eta t} \quad t> t_6.
\end{cases}  
\end{align*}

\begin{lemma} Define $\overline{\phi}(t), \underline{\phi}(t), \overline{\varphi}(t), \underline{\psi}(t), \overline{\psi}(t), \underline{\psi}(t)$ as above then the following are true for $t\in \R$
\begin{enumerate}
    \item $\underline{\phi}(t)\le \overline{\phi}(t),$
    \item$\underline{\varphi}(t)\le \overline{\varphi}(t),$
    \item $\underline{\psi}(t)\le \overline{\psi}(t).$
\end{enumerate}
\end{lemma}
\begin{proof}
The proof relies on the fact that for all $\lambda \in \R$ $P_1(\lambda)>P_2(\lambda), P_3(\lambda)> P_4(\lambda), P_5(\lambda)> P_6(\lambda)$. This gives $\lambda_1<\lambda_2, \lambda_3<\lambda_4, \lambda_5<\lambda_6.$ By the implicit function theorem we see $\eta_1<\eta_2, \eta_3<\eta_4, \eta_5<\eta_6.$ Moreover, $t_1\le t_4, t_2\le t_5, t_3\le t_6.$ Using this information it is easy to see the result follows.    
\end{proof}
The critical wave speed is defined by 
\[
c^*\max \left\{\begin{array}{l}2\sqrt{\left|-\mu_1, (\mu_2-\mu_1)M_2/M_1+(\mu_3-\mu_1)M_3/M_1\right|}, 2\sqrt{\mu_1},\\
2\sqrt{\left|\beta B M_2\mu_1-(\gamma +\mu_2)\right|, 2}, 2\sqrt{\gamma +\mu_2}, 2\sqrt{\left|\gamma M_2/M_3 -\mu_3\right|}, 2\sqrt{\mu_3}\end{array} \right\}\].

\subsection{Upper Solutions}
In this section, we will construct an appropriate set of upper solutions for system (\ref{wave 1}) via the iterative operator. We begin by finding explicit super solutions.

\begin{claim}\label{super} Let $c\ge c^*$ and $\mathcal{R}_0>1$, and $\overline{\phi},  \underline{\phi}, \overline{\varphi}, \underline{\varphi}, \overline{\psi},  \underline{\psi}$ as above, then $\left(\overline{\phi}, \overline{\varphi}, \overline{\psi}\right)^T$ is a  super solution of system  (\ref{wave 1}). 
\end{claim}
\begin{proof}
 The proof for equation 1 will be done in cases.

 \noindent {\textbf Case 1: $t\le t_1-r_1$}
\begin{align*}
&\phi''(t)-c\phi'(t+r_1)-\mu_1\phi(t+r_1)+(\mu_2-\mu_1)\varphi(t+r_1)+(\mu_3-\mu_1)\psi(t+r_1)\\
&=k_1\eta_1^2e^{\eta_1 t}-ck_1\eta_1e^{\eta_1 (t+r_1)}-\mu_1e^{\eta_1 (t+r_1)}+(\mu_2-\mu_1)k_2e^{\eta_3 (t+r_1)}+(\mu_3-\mu_1)k_3e^{\eta_5 (t+r_1)}\\
&\le \left(\eta_1^2-c\eta_1e^{r_1\eta_1}-\mu_1e^{r_1\eta_1}+(\mu_2-\mu_1)M_2/M_1e^{r_1\eta_1}+(\mu_3-\mu_1)M_3/M_1e^{r_1\eta_1}\right)k_1e^{\eta_1 t}\\
&\Delta_1(\eta_1)k_1e^{\eta_1 t}=0.
\end{align*}
\medskip \noindent   {\textbf Case 2: $t_1-r_1<t\le t_1$}
\begin{align*}
 &\phi''(t)-c\phi'(t+r_1)-\mu_1\phi(t+r_1)+(\mu_2-\mu_1)\varphi(t+r_1)+(\mu_3-\mu_1)\psi(t+r_1)\\
 &\le k_1\eta_1^2e^{\eta_1 t}+c\eta\varepsilon_1e^{-\eta(t+r_1)}-\mu_1\left(k_1+\varepsilon_1e^{-\eta(t+r_1)}\right)+(\mu_2-\mu_1)M_2+(\mu_3-\mu_1)M_3\\
 &\le k_1\eta_1^2+c\eta\varepsilon_1e^{-\eta(t+r_1)}-\mu_1\left(k_1+\varepsilon_1e^{-\eta(t+r_1)}\right)+(\mu_2-\mu_1)M_2+(\mu_3-\mu_1)M_3=I_1(\eta).
\end{align*}
\noindent Now, $I_1(0)=k_1\eta_1^2-\mu_1\left(k_1+\varepsilon_1\right)+(\mu_2-\mu_1)M_2+(\mu_3-\mu_1)M_3<0$
\noindent by system (\ref{solcond}). Thus, there is a $\eta_1^*>0$ such that 
\[\phi''(t)-c\phi'(t+r_1)-\mu_1\phi(t+r_1)+(\mu_2-\mu_1)\varphi(t+r_1)+(\mu_3-\mu_1)\psi(t+r_1)< 0, \ \forall \eta\in (0,\eta_1^*).\]

\medskip \noindent   {\textbf Case 3: $t_1<t$}
 \begin{align*}
 &\phi''(t)-c\phi'(t+r_1)-\mu_1\phi(t+r_1)+(\mu_2-\mu_1)\varphi(t+r_1)+(\mu_3-\mu_1)\psi(t+r_1)\\
 &\le \varepsilon_1\eta^2e^{-\eta t}+c\eta\varepsilon_1e^{-\eta(t+r_1)}-\mu_1\left(k_1+\varepsilon_1e^{-\eta(t+r_1)}\right)+(\mu_2-\mu_1)M_2+(\mu_3-\mu_1)M_3\\
 &\le\left( \varepsilon_1\eta^2+c\eta\varepsilon_1e^{-\eta r_1}\right)e^{-\eta t}-\mu_1\left(k_1+\varepsilon_1e^{-\eta(t+r_1)}\right)+(\mu_2-\mu_1)M_2+(\mu_3-\mu_1)M_3=I_2(\eta).
\end{align*}
\noindent Now, $I_2(0)=-\mu_1\left(k_1+\varepsilon_1\right)+(\mu_2-\mu_1)M_2+(\mu_3-\mu_1)M_3<0$
\noindent by system (\ref{solcond}). Thus, there is a $\eta_2^*>0$ such that 
\[\phi''(t)-c\phi'(t+r_1)-\mu_1\phi(t+r_1)+(\mu_2-\mu_1)\varphi(t+r_1)+(\mu_3-\mu_1)\psi(t+r_1)< 0, \ \forall \eta\in (0,\eta_2^*).\]

 \medskip \noindent The proof for equation 2 will also be split into cases. 

\noindent {\textbf Case 1: $t\le t_2-r_2$}
\begin{align*}
& {\varphi}''(t)-c{\varphi}'(t+r_2)-(\mu_2+\gamma)\varphi(t+r_2)
+\frac{\beta \left(B/\mu_1-\phi(t+r_2)-\varphi(t+r_2)-\psi(t+r_2)\right)\varphi(t+r_2-r_4)}{1+\alpha \varphi(t+r_2-r_4) }\\
&\le \left(\eta_3^2-c\eta_2e^{r_3\eta_3}+\beta B M_2e^{r_3\eta_3}-(\mu_2+\gamma)e^{r_3\eta_3}\right)k_2e^{r_3t}=\Delta_3(\eta_3)k_2e^{r_3t}=0.
\end{align*}
\medskip \noindent   {\textbf Case 2: $t_2-r_2<t\le t_2$}
\begin{align*}
& {\varphi}''(t)-c{\varphi}'(t+r_2)-(\mu_2+\gamma)\varphi(t+r_2)
+\frac{\beta \left(B/\mu_1-\phi(t+r_2)-\varphi(t+r_2)-\psi(t+r_2)\right)\varphi(t+r_2-r_4)}{1+\alpha \varphi(t+r_2-r_4) }\\
&\le k_2\eta_3^2+c\eta\varepsilon_3 e^{-\eta(t+r_3)}+\frac{\beta \left(B/\mu_1 -k_1+\varepsilon_2e^{-\eta(t+r_3)}-k_2-\varepsilon_3-k_3+\varepsilon_6e^{-\eta(t+r_3)}\right)M_2 }{1+\alpha M_2}\\
&-(\mu_2+\gamma)\left(k_2+\varepsilon_3e^{-\eta(t+r_3)}\right)=I_3(\eta).
\end{align*}

\noindent Now, \[I_3(0)=k_2\eta_3^2+\frac{\beta \left(B/\mu_1 -k_1+\varepsilon_2-k_2-\varepsilon_3-k_3+\varepsilon_6\right)M_2 }{1+\alpha M_2}-(\mu_2+\gamma)\left(k_2+\varepsilon_3\right)<0\]
\noindent by system (\ref{solcond}). Thus, there is a $\eta_3^*>0$ such that
\[{\varphi}''(t)-c{\varphi}'(t+r_2)-(\mu_2+\gamma)\varphi(t+r_2)
+\frac{\beta \left(B/\mu_1-\phi(t+r_2)-\varphi(t+r_2)-\psi(t+r_2)\right)\varphi(t+r_2-r_4)}{1+\alpha \varphi(t+r_2-r_4) }<0\]
$\forall \eta\in (0,\eta_3^*).$

\medskip \noindent   {\textbf Case 3: $t_2<t$}
\begin{align*}
& {\varphi}''(t)-c{\varphi}'(t+r_2)-(\mu_2+\gamma)\varphi(t+r_2)
+\frac{\beta \left(B/\mu_1-\phi(t+r_2)-\varphi(t+r_2)-\psi(t+r_2)\right)\varphi(t+r_2-r_4)}{1+\alpha \varphi(t+r_2-r_4) }\\
&\le k_2\eta^2+c\eta\varepsilon_3 e^{-\eta(t+r_3)}+\frac{\beta \left(B/\mu_1 -k_1+\varepsilon_2e^{-\eta(t+r_3)}-k_2-\varepsilon_3-k_3+\varepsilon_6e^{-\eta(t+r_3)}\right)M_2 }{1+\alpha M_2}\\
&-(\mu_2+\gamma)\left(k_2+\varepsilon_3e^{-\eta(t+r_3)}\right)=I_4(\eta).
\end{align*}

\noindent Now, \[I_4(0)=\frac{\beta \left(B/\mu_1 -k_1+\varepsilon_2-k_2-\varepsilon_3-k_3+\varepsilon_6\right)M_2 }{1+\alpha M_2}-(\mu_2+\gamma)\left(k_2+\varepsilon_3\right)<0\]
\noindent by system (\ref{solcond}). Thus, there is a $\eta_4^*>0$ such that
\[{\varphi}''(t)-c{\varphi}'(t+r_2)-(\mu_2+\gamma)\varphi(t+r_2)
+\frac{\beta \left(B/\mu_1-\phi(t+r_2)-\varphi(t+r_2)-\psi(t+r_2)\right)\varphi(t+r_2-r_4)}{1+\alpha \varphi(t+r_2-r_4) }<0\]
$\forall \eta\in (0,\eta_4^*).$

\medskip \noindent The proof for equation 3 will also be split into cases. 

\noindent {\textbf Case 1: $t\le t_3-r_3$}
\begin{align*}
& {\psi}''(t)-c{\psi}'(t+r_3)-\mu_3 \psi(t+r_3)+\gamma\psi(t+r_3)\\
&\le \left(\eta_5^2-ck_2\eta_5e^{\eta_5r_3}+\frac{\gamma M_2}{M_3}-\mu_3e^{\eta_5r_3}\right)k_3e^{\eta_5 t}\\
&=\Delta_5(\eta_5)k_3e^{\eta_5 t}=0.
\end{align*}
\medskip \noindent   {\textbf Case 2: $t_3-r_3<t\le t_3$}

\begin{align*}
&{\psi}''(t)-c{\psi}'(t+r_3)-\mu_3 \psi(t+r_3)+\gamma\psi(t+r_3)\\
&\le \eta_5^2+c\varepsilon_5\eta e^{-\eta (t-r_3)}+\gamma M_2-\mu_3\left(k_3+\varepsilon_5e^{-\eta (t-r_3)}\right)=I_5(\eta).
\end{align*}

\noindent Now, \[I_5(0)=\eta_5^2+\gamma M_2-\mu_3\left(k_3+\varepsilon_5\right)<0\]
\noindent by system (\ref{solcond}). Thus, there is a $\eta_5^*>0$ such that
\[{\psi}''(t)-c{\psi}'(t+r_3)-\mu_3 \psi(t+r_3)+\gamma\psi(t+r_3)<0\]
$\forall \eta\in (0,\eta_5^*).$

\medskip \noindent   {\textbf Case 3: $t_3<t$}

\begin{align*}
&{\psi}''(t)-c{\psi}'(t+r_3)-\mu_3 \psi(t+r_3)+\gamma\psi(t+r_3)\\
&\le \eta^2 e^{-\eta t}+c\varepsilon_5\eta e^{-\eta (t-r_3)}+\gamma M_2 -\mu_3\left(k_3+\varepsilon_5e^{-\eta (t-r_3)}\right)=I_6(\eta).
\end{align*}

\noindent Now, \[I_6(0)=\gamma M_2-\mu_3\left(k_3+\varepsilon_5\right)<0\]
\noindent by system (\ref{solcond}). Thus, there is a $\eta_6^*>0$ such that
\[{\psi}''(t)-c{\psi}'(t+r_3)-\mu_3 \psi(t+r_3)+\gamma\psi(t+r_3)<0\]
$\forall \eta\in (0,\eta_6^*).$

\noindent Take $\eta\in \left(0, \min\{\eta_1^*,\eta_2^*,\eta_3^*,\eta_{4}^*, \eta_{5}^*,\eta_{6}^*\}\right),$ then the result is proven.

\end{proof}

\begin{corollary}\label{quasiup}
 Let $c\ge c^*$ and $\mathcal{R}_0>1$, and $\overline{\phi_0},   \overline{\varphi_0}, \overline{\psi_0},$ to be the super solution as in Claim \ref{super} , then $F\left(\overline{\phi_0}, \overline{\varphi_0}, \overline{\psi_0}\right)$ is a  quasi-upper solution of system  (\ref{wave 1}).    
\end{corollary}

\begin{corollary}
 Let $c\ge c^*$ and $\mathcal{R}_0>1$, and $\overline{\phi_1},   \overline{\varphi_1}, \overline{\psi_1},$ to be the quasi-upper solution as in Claim \ref{quasiup} , then $F\left(\overline{\phi_1}, \overline{\varphi_1}, \overline{\psi_1}\right)$ is an upper solution of system  (\ref{wave 1}).    
\end{corollary}
\subsection{Lower Solutions}
In this section, we will construct an appropriate set of quasi-lower solutions for system \ref{wave 1}.

\begin{claim}\label{sub} Let $c\ge c^*$ and $\mathcal{R}_0>1$, and $\overline{\phi},  \underline{\phi}, \overline{\varphi}, \underline{\varphi}, \overline{\psi},  \underline{\psi}$ as above, then $\left(\underline{\phi}, \underline{\varphi}, \underline{\psi}\right)^T$ is a  subsolution of  \ref{wave 1}. 
\end{claim}

\begin{proof}
 The proof for equation 1 will be done in cases.

 \noindent {\textbf Case 1: $t\le t_2-r_1$}
\begin{align*}
&\phi''(t)-c\phi'(t+r_1)-\mu_1\phi(t+r_1)+(\mu_2-\mu_1)\varphi(t+r_1)+(\mu_3-\mu_1)\psi(t+r_1)\\
&=\Delta_2(\eta_2)k_1e^{\eta_2}+(\mu_2-\mu_1)\varphi(t+r_1)+(\mu_3-\mu_1)\psi(t+r_1)\\
&=(\mu_2-\mu_1)\varphi(t+r_1)+(\mu_3-\mu_1)\psi(t+r_1)>0.
\end{align*}

\medskip \noindent   {\textbf Case 2: $t_1-r_1<t\le t_2$}
\begin{align*}
&\phi''(t)-c\phi'(t+r_1)-\mu_1\phi(t+r_1)+(\mu_2-\mu_1)\varphi(t+r_1)+(\mu_3-\mu_1)\psi(t+r_1)\\
&=\frac{k_1e^{\eta_2}}{M}e^{\eta_2 t}-c\varepsilon_4\eta e^{-\eta(t+r_1)}-\mu_1\left(k_1-\varepsilon_2e^{-\eta(t+r_1)}\right)+(\mu_2-\mu_1)\left(k_2-\varepsilon_4e^{-\eta(t+r_1)}\right)\\
&+(\mu_3-\mu_1)\left(k_3-\varepsilon_6e^{-\eta(t+r_1)}\right)=I_7(\eta).
\end{align*}

\noindent Now, \[I_7(0)=\frac{k_1e^{\eta_2}}{M}e^{\eta_2 t}-\mu_1\left(k_1-\varepsilon_2\right)+(\mu_2-\mu_1)\left(k_2-\varepsilon_4\right)+(\mu_3-\mu_1)\left(k_3-\varepsilon_6\right)>0\]
\noindent by system (\ref{solcond}). Thus, there is a $\eta_7^*>0$ such that
\[\phi''(t)-c\phi'(t+r_1)-\mu_1\phi(t+r_1)+(\mu_2-\mu_1)\varphi(t+r_1)+(\mu_3-\mu_1)\psi(t+r_1)>0, \
\forall \eta\in (0,\eta_5^*).\]

\medskip \noindent   {\textbf Case 3: $t_2<t$}
\begin{align*}
&\phi''(t)-c\phi'(t+r_1)-\mu_1\phi(t+r_1)+(\mu_2-\mu_1)\varphi(t+r_1)+(\mu_3-\mu_1)\psi(t+r_1)\\
&=k_1\eta^2e^{-\eta t}-c\varepsilon_4\eta e^{-\eta(t+r_1)}-\mu_1\left(k_1-\varepsilon_2e^{-\eta(t+r_1)}\right)+(\mu_2-\mu_1)\left(k_2-\varepsilon_4e^{-\eta(t+r_1)}\right)\\
&+(\mu_3-\mu_1)\left(k_3-\varepsilon_6e^{-\eta(t+r_1)}\right)=I_8(\eta).
\end{align*}

\noindent Now, \[I_8(0)=-\mu_1\left(k_1-\varepsilon_2\right)+(\mu_2-\mu_1)\left(k_2-\varepsilon_4\right)+(\mu_3-\mu_1)\left(k_3-\varepsilon_6\right)>0\]
\noindent by system (\ref{solcond}). Thus, there is a $\eta_8^*>0$ such that
\[\phi''(t)-c\phi'(t+r_1)-\mu_1\phi(t+r_1)+(\mu_2-\mu_1)\varphi(t+r_1)+(\mu_3-\mu_1)\psi(t+r_1)>0, \
\forall \eta\in (0,\eta_6^*).\]

 \medskip \noindent The proof for equation 2 will also be split into cases. 

\noindent {\textbf Case 1: $t\le t_4-r_2$}

\begin{align*}
& {\varphi}''(t)-c{\varphi}'(t+r_2)-(\mu_2+\gamma)\varphi(t+r_2)
+\frac{\beta \left(B/\mu_1-\phi(t+r_2)-\varphi(t+r_2)-\psi(t+r_2)\right)\varphi(t+r_2-r_4)}{1+\alpha \varphi(t+r_2-r_4) }\\
&= \Delta_4(\eta_4)\frac{k_2e^{r_2t}}{M}+\frac{\beta \left(B/\mu_1-\phi(t+r_2)-\varphi(t+r_2)-\psi(t+r_2)\right)\varphi(t+r_2-r_4)}{1+\alpha \varphi(t+r_2-r_4) }\\
&=\frac{\beta \left(B/\mu_1-\phi(t+r_2)-\varphi(t+r_2)-\psi(t+r_2)\right)\varphi(t+r_2-r_4)}{1+\alpha \varphi(t+r_2-r_4) }>0.
\end{align*}

\medskip \noindent   {\textbf Case 2: $t_4-r_2<t\le t_4$}
\begin{align*}
& {\varphi}''(t)-c{\varphi}'(t+r_2)-(\mu_2+\gamma)\varphi(t+r_2)
+\frac{\beta \left(B/\mu_1-\phi(t+r_2)-\varphi(t+r_2)-\psi(t+r_2)\right)\varphi(t+r_2-r_4)}{1+\alpha \varphi(t+r_2-r_4) }\\
&\ge \frac{k_2e^{\eta_4 t}}{M}-c\varepsilon_5\eta e^{-\eta(t+r_2)}-(\mu_2+\gamma)\left(k_2-\varepsilon_5 e^{-\eta(t+r_2)}\right)\\
&+\frac{\beta \left(B/\mu_1-M_1-k_2+\varepsilon_5 e^{-\eta(t+r_2)}-M3\right)\left(k_2-\varepsilon_5 e^{-\eta(t+r_2)}\right)}{1+\alpha \left(k_2-\varepsilon_5 e^{-\eta(t+r_2)}\right) }=I_9(\eta).
\end{align*}
\noindent Now, \[I_9(0)=\frac{k_2e^{\eta_3 t}}{M}-(\mu_2+\gamma)\left(k_2-\varepsilon_5\right)+\frac{\beta \left(B/\mu_1-M_1-k_2+\varepsilon_5 -M3\right)\left(k_2-\varepsilon_5 \right)}{1+\alpha \left(k_2-\varepsilon_5\right) }>0\]
\noindent by system (\ref{solcond}). Thus, there is a $\eta_9^*>0$ such that
\[{\varphi}''(t)-c{\varphi}'(t+r_2)-(\mu_2+\gamma)\varphi(t+r_2)
+\frac{\beta \left(B/\mu_1-\phi(t+r_2)-\varphi(t+r_2)-\psi(t+r_2)\right)\varphi(t+r_2-r_4)}{1+\alpha \varphi(t+r_2-r_4) }
\]
$\forall \eta\in (0,\eta_9^*).$

\medskip \noindent   {\textbf Case 3: $t_4<t$}
\begin{align*}
& {\varphi}''(t)-c{\varphi}'(t+r_2)-(\mu_2+\gamma)\varphi(t+r_2)
+\frac{\beta \left(B/\mu_1-\phi(t+r_2)-\varphi(t+r_2)-\psi(t+r_2)\right)\varphi(t+r_2-r_4)}{1+\alpha \varphi(t+r_2-r_4) }\\
&\ge \varepsilon_5 \eta^2 e^{-\eta t}-c\varepsilon_5\eta e^{-\eta(t+r_2)}-(\mu_2+\gamma)\left(k_2-\varepsilon_5 e^{-\eta(t+r_2)}\right)\\
&+\frac{\beta \left(B/\mu_1-M_1-k_2+\varepsilon_5 e^{-\eta(t+r_2)}-M3\right)\left(k_2-\varepsilon_5 e^{-\eta(t+r_2)}\right)}{1+\alpha \left(k_2-\varepsilon_5 e^{-\eta(t+r_2)}\right) }=I_{10}(\eta).
\end{align*}

\noindent Now, \[I_{10}(0)=-(\mu_2+\gamma)\left(k_2-\varepsilon_5 \right)+\frac{\beta \left(B/\mu_1-M_1-k_2+\varepsilon_5 -M3\right)\left(k_2-\varepsilon_5 \right)}{1+\alpha \left(k_2-\varepsilon_5 \right) }>0\]
\noindent by system (\ref{solcond}). Thus, there is a $\eta_{10}^*>0$ such that
\[{\varphi}''(t)-c{\varphi}'(t+r_2)-(\mu_2+\gamma)\varphi(t+r_2)
+\frac{\beta \left(B/\mu_1-\phi(t+r_2)-\varphi(t+r_2)-\psi(t+r_2)\right)\varphi(t+r_2-r_4)}{1+\alpha \varphi(t+r_2-r_4) }>0
\]
$\forall \eta\in (0,\eta_{10}^*).$

\medskip \noindent The proof for equation 3 will also be split into cases. 

\noindent {\textbf Case 1: $t\le t_6-r_3$}

\begin{align*}
&{\psi}''(t)-c{\psi}'(t+r_3)-\mu_3 \psi(t+r_3)+\gamma\psi(t+r_3)\\
&= \Delta_6(\eta_6)\frac{k_3 e^{\eta_6 t}}{M}+\gamma\psi(t+r_3)=\gamma\psi(t+r_3)>0.
\end{align*}

\medskip \noindent   {\textbf Case 2: $t_6-r_3<t\le t_6$}

\begin{align*}
&{\psi}''(t)-c{\psi}'(t+r_3)-\mu_3 \psi(t+r_3)+\gamma\psi(t+r_3)\\
&\ge\frac {k_3\eta_6^2}{M} e^{\eta_6 t}-c\varepsilon_6\eta e^{-\eta (t-r_3)}+\gamma\left(k_2-\varepsilon_4e^{-\eta (t-r_3)}\right) -\mu_3\left(k_3-\varepsilon_6e^{-\eta (t-r_3)}\right)=I_{11}(\eta).
\end{align*}

\noindent Now, \[I_{11}(0)=\frac {k_3\eta_6^2}{M} e^{\eta_6 t}+\gamma\left(k_2-\varepsilon_4\right) -\mu_3\left(k_3-\varepsilon_6\right)>0\]
\noindent by system (\ref{solcond}). Thus, there is a $\eta_{11}^*>0$ such that
\[{\psi}''(t)-c{\psi}'(t+r_3)-\mu_3 \psi(t+r_3)+\gamma\psi(t+r_3)>0
\]
$\forall \eta\in (0,\eta_{11}^*).$

\medskip \noindent   {\textbf Case 3: $t_6<t$}
\begin{align*}
&{\psi}''(t)-c{\psi}'(t+r_3)-\mu_3 \psi(t+r_3)+\gamma\psi(t+r_3)\\
&\ge \eta^2 \varepsilon_6 e^{-\eta t}-c\varepsilon_6\eta e^{-\eta (t-r_3)}+\gamma\left(k_2-\varepsilon_4e^{-\eta (t-r_3)}\right) -\mu_3\left(k_3-\varepsilon_6e^{-\eta (t-r_3)}\right)=I_{12}(\eta).
\end{align*}

\noindent Now, \[I_{12}(0)=\gamma\left(k_2-\varepsilon_4\right) -\mu_3\left(k_3-\varepsilon_6\right)>0\]
\noindent by system (\ref{solcond}). Thus, there is a $\eta_{12}^*>0$ such that
\[{\psi}''(t)-c{\psi}'(t+r_3)-\mu_3 \psi(t+r_3)+\gamma\psi(t+r_3)>0
\]
$\forall \eta\in (0,\eta_{12}^*).$

\noindent Take $\eta\in \left(0, \min\{\eta_7^*,\eta_8^*,\eta_9^*,\eta_{10}^*, \eta_{11}^*,\eta_{12}^*\}\right),$ then the result is proven.
\end{proof}

\begin{corollary}\label{quasilow}
 Let $c\ge c^*$ and $\mathcal{R}_0>1$, and $\underline{\phi_0},   \underline{\varphi_0}, \underline{\psi_0},$ to be the sub solution as in Claim \ref{sub} , then $F\left(\underline{\phi_0}, \underline{\varphi_0}, \underline{\psi_0}\right)$ is a  quasi-lower solution of system  (\ref{wave 1}).    
\end{corollary}

\begin{corollary}
 Let $c\ge c^*$ and $\mathcal{R}_0>1$, and $\underline{\phi_1},   \underline{\varphi_1}, \underline{\psi_1},$ to be the quasi-lower solution as in Claim \ref{quasilow} , then $F\left(\underline{\phi_1}, \underline{\varphi_1}, \underline{\psi_1}\right)$ is a lower solution of system  (\ref{wave 1}).    
\end{corollary}
\begin{corollary}\label{last}
Assume that $c>c^* $ is given. Then, the system (\ref{SIR4}) has a traveling wave solution $(S(x,t),I(x,t), R(x,t)^T=({\phi(x+ct)}, {\varphi(x+ct),{\psi(x+ct)} )^T}$ for sufficiently small delays $\tau_1,\tau_2, \tau_3, \tau_4$.
\end{corollary}

\subsection{Statements and Declarations}
There are no competing declarations or sources of funding for this work.
\newpage 
\bibliographystyle{amsplain}

\end{document}